\documentclass[a4paper,11pt,reqno, english]{amsart}

\usepackage{anysize,color}
\usepackage[utf8]{inputenc}
\usepackage[dvipsnames]{xcolor}
\usepackage[colorlinks=true,urlcolor=blue,linkcolor={YellowOrange},citecolor={YellowOrange}]{hyperref}  
\usepackage{float}
\usepackage{amsfonts,amssymb,amsmath}
\usepackage{amsthm}    
\usepackage{url}
\usepackage{paralist}
\usepackage{tikz}
    \usetikzlibrary{decorations.markings,arrows,shapes.callouts}
\usepackage{rotating} 
\usepackage[mathscr]{euscript}
\usepackage{verbatim}
\usepackage{tikz-cd}
\usepackage{mathtools}
\usepackage{dsfont}
\usepackage{amsrefs}
\usepackage{thmtools}
\usepackage{thm-restate}

\usepackage{charter}

\newtheorem{theorem}{Theorem}[section]
\newtheorem{proposition}[theorem]{Proposition}
\newtheorem{lemma}[theorem]{Lemma}

\theoremstyle{definition}
\newtheorem{definition}[theorem]{Definition}
\theoremstyle{remark}

\newtheorem{remark}[theorem]{Remark} 

\newcommand{\R}{\mathbb{R}}
\newcommand{\C}{\mathbb{C}}
\newcommand{\FF}{\mathbb{F}}

\newcommand{\Z}{\mathbb{Z}}

\newcommand{\F}{\mathcal{F}}
\newcommand{\G}{\mathcal{G}}

\def\keywords{\xdef\@thefnmark{}\@footnotetext}

\DeclareMathOperator{\proj}{\mathrm{proj}}

\DeclareMathOperator{\conv}{\mathrm{conv}}

\DeclareMathOperator{\linspan}{\mathrm{span}}

\tikzcdset{row sep/my_size/.initial=-2mm}
\tikzcdset{row sep/my_size_small/.initial=-3mm}
\newcommand{\xDownarrow}[1]{%
  {\left\Downarrow\vbox  to #1{}\right.\kern-\nulldelimiterspace}
}

\begin{document}

\title{A necessary and sufficient condition for $k$-transversals}


\author[McGinnis]{Daniel McGinnis} 
\address{Department of Mathematics, Princeton University, USA}
\email{dm7932@princeton.edu}

\author[Sadovek]{Nikola Sadovek}
\address[NS]{Max Planck Institute of Molecular Cell Biology and Genetics, Dresden, Germany \newline
Center for Systems Biology Dresden, Dresden, Germany \newline
Faculty of Mathematics, Technische Universit\"at Dresden, Dresden, Germany}
\email{sadovek@mpi-cbg.de,~nikolasdvk@gmail.com}

\thanks{\hspace{-4mm}\textit{2020 Mathematics Subject Classification.} 52A35. \\
The research of D.\ McGinnis is funded by the National Science Foundation (NSF) under award no. 2402145.\\
The research of N.\ Sadovek is funded by the Deutsche Forschungsgemeinschaft (DFG, German Research Foundation) under Germany's Excellence Strategy--The Berlin Mathematics Research Center MATH+ (EXC-2046/1, project ID 390685689, BMS Stipend).}


\begin{abstract}
    We solve a long-standing open problem posed by Goodman \& Pollack in 1988 by establishing a necessary and sufficient condition for a family of convex sets in $\R^d$ to admit a $k$-transversal for any $0 \le k \le d-1$. This result is a common generalization of Helly's theorem ($k=0$) and the Goodman-Pollack-Wenger theorem ($k=d-1$). Additionally, we obtain an analogue in the complex setting by characterizing the existence of a complex $k$-transversal to a family of convex sets in $\C^d$, extending the work of McGinnis ($k=d-1$). Our approach is topological and employs a Borsuk-Ulam-type theorem on Stiefel manifolds. Finally, we demonstrate how our results imply the central transversal theorems of \v{Z}ivaljevi\'c-Vre\'cica and Dol'nikov in the real case and of Sadovek-Sober\'on in the complex case.
\end{abstract}

\maketitle


\section{Introduction}

Helly's theorem \cite{helly1923mengen} is a cornerstone result in discrete geometry that has motivated many interesting directions of research that continue to be explored to this day \cites{AmentaHelly2017,HolmsenChapter}. It states that if a finite family $\F$ of convex sets in $\mathbb{R}^d$ has the property that every subfamily of $d+1$ or fewer sets have a nonempty intersection, then the intersection of all the sets in $\F$ is nonempty.  One question explored early on by Vincensini \cite{vincensini1935figures} is whether there exists a Helly-type condition on families $\F$ of convex sets in $\mathbb{R}^d$ that guarantees the existence of a \textit{$k$-transversal} to $\F$. A $k$-transversal to $\F$ is a $k$-dimensional affine subspace (or $k$-flat) of $\mathbb{R}^d$ that intersects each set in $\F$. In particular, Helly's theorem provides a necessary and sufficient condition for such families $\F$ to have a $0$-transversal. Vincensini asked if there exists a constant $r(k,d)$ such that the following statement holds for finite families $\F$ of convex sets in $\mathbb{R}^d$: if every subfamily of $r(k,d)$ sets has a $k$-transversal, then $\F$ has a $k$-transversal. However, this was proven to be false by Santal\'o \cite{santalotheorem1940}.

The first substantial progress toward determining a condition that guarantees the existence of a $k$-transversal for $k>0$ was made by Hadwiger \cite{HadwigerLines} in the case that $k=1$, $d=2$, and the family $\F$ consists of \textit{pairwise disjoint} convex sets. The key observation is that a $1$-transversal to $\F$ determines a \textit{linear ordering} on the sets of $\F$.

\begin{theorem}[Hadwiger \cite{HadwigerLines}]
    A finite family of pairwise disjoint convex sets in $\mathbb{R}^2$ has a $1$-transversal if and only if the sets in the family can be linearly ordered such that any three sets have a $1$-transversal consistent with the ordering.
\end{theorem}

In 1990, a series of three papers among the authors Goodman, Pollack and Wenger \cites{GoodmanHadwiger1988,WengerIntersecting1990,PollackNecessary1990} culminated into the Goodman-Pollack-Wenger theorem, a generalization of Hadwiger's theorem for $(d-1)$-transversals to finite families of convex sets in $\mathbb{R}^d$ with no disjointness condition on the sets. See Section \ref{sec:g-p-w-(d-1)} for more details. This is a celebrated result in geometric transversal theory that has since been expanded upon in many different ways \cites{Anderson1996Oriented,arocha2002separoids,ArochaColorful2008,HolmsenColored2016,cheong2024new}. See also \cites{GoodmanGeometric1993,HolmsenChapter} for surveys on related results in geometric transversal theory.

However, up to this point, there was no known condition that guarantees the existence of a $k$-transversal to a family of convex sets in $\mathbb{R}^d$ for $0< k< d-1$. The question of the existence of such a condition was initially raised by Goodman \& Pollack \cite{GoodmanHadwiger1988} and later also asked in \cite{PollackNecessary1990,HolmsenChapter}.
Arocha, Bracho, Montejano, Oliveros \& Strausz \cite{arocha2002separoids} provided a sufficient condition for the existence of a \emph{virtual} $k$-transversal, which is a broader notion coinciding with the existence of a $k$-transversal in the case when the set family has cardinality $k+2$.
Moreover, we note that a \textit{complex analogue} of the Goodman-Pollack-Wenger theorem was proven by McGinnis in \cite{mcginnis2023necessary}, and this was mistakenly claimed to imply a necessary and sufficient condition for the existence of $(2d-2)$-transversals in $\mathbb{R}^{2d}$. However, this is incorrect and was later corrected in \cite{mcginnis2023complex}. 

The main result of this paper resolves a long-standing open problem posed by Goodman \& Pollack \cite{GoodmanHadwiger1988} by establishing a necessary and sufficient condition for a finite family of convex sets in a $d$-dimensional space to admit a $k$-transversal, for any $0 \leq k \leq d-1$. {The condition is described as follows:}

{
\begin{definition} \label{def:R-dependency-consistent-with-tuples}
    Let $0 \le k < d$ be integers, $\F$ a finite family of convex sets in $\R^d$, and $P \subseteq \R^k$ a finite set of points.
    We say that $\F$ is \emph{$\R$-dependency consistent with $(d-k)$-tuples} in $P$ if there is a map $\phi \colon \F \to P$ such that for any subset $\F' \subseteq \F$ with $|\F'| \leq (k+1)(d-k) +1$ and any $d-k$ affine dependencies
    \[
        \sum_{F \in \F'} a^{(i)}_F = 0 \in \R, \sum_{F \in \F'} a^{(i)}_F \phi(F) = 0 \in \R^k, \hspace{11mm} \textrm{for }i=1, \dots, d-k,
    \]
    which are not all trivial, there exist points $q_F \in F$ and real numbers $r_F \ge 0$, such that the affine dependencies
    \[
        \sum_{F \in \F'} r_Fa^{(i)}_F = 0 \in \R, \sum_{F \in \F'} r_Fa^{(i)}_F q_F = 0 \in \R^d, \hspace{7mm} \textrm{for }i=1, \dots, d-k,
    \]
    are not all trivial, i.e. some $r_Fa^{(i)}_F$ is nonzero.
\end{definition}}

{ The motivation for Definition \ref{def:R-dependency-consistent-with-tuples} is explained in further detail in Section \ref{sec:g-p-w-(d-1)}. We may now state the main result of the paper:}

\begin{theorem} \label{thm:main-R}
    Let $0 \le k < d$ be integers and $\F$ a finite family of convex sets in $\R^d$. Then, there exists a $k$-transversal to $\F$ if and only if $\F$ is $\R$-dependency consistent with $(d-k)$-tuples in a finite set of points $P \subseteq \R^k$.
\end{theorem}

{
The new condition is linear-algebraic, but its two extreme cases ($k=0$ and $k=d-1$) coincide with more geometric conditions from Helly's theorem and the Goodman-Pollack-Wenger theorem, respectively. This is explained in Remark \ref{rem:k=0-and-k=d-1} in more detail. Thus, the $k=0$ case of Theorem \ref{thm:main-R} is Helly's theorem \cite{helly1923mengen}, while $k=d-1$ recovers Goodman-Pollack-Wenger theorem \cites{GoodmanHadwiger1988,WengerIntersecting1990,PollackNecessary1990}.} Moreover, {we show in Section \ref{sec:application} that} our result implies the \textit{central transversal theorem} obtained independently by \v{Z}ivaljevi\'c \& Vre\'cica \cite{zivaljevic1990extension} and Dol'nikov \cite{Dolnikov:1992ut}, which has Rado's centerpoint theorem \cite{Rado:1946ud} and the ham sandwich theorem \cite{Steinhaus1938} as the extremal cases. See Figure \ref{fig:implications} for an illustration.

\begin{figure}[h]
    \centering
    \begin{equation*}
    \begin{tikzcd}[row sep = -4.5mm, column sep = -1mm]
        \text{Helly} & \xLeftarrow{\phantom{aaa}} & \text{Theorem \ref{thm:main-R}} & \xRightarrow{\phantom{aaa}} & \text{Goodman-Pollack-Wenger}\\
        (k=0) & & (0 \le k \le d-1) & & (k=d-1) \\
        \phantom{\xDownarrow{3mm}} & & {} & & {}\\
        \xDownarrow{3mm} & & \xDownarrow{3mm} & & \xDownarrow{3mm} \\
        \phantom{\xDownarrow{3mm}} & & {} & & {}\\
        \text{Rado's centerpoint} & \xLeftarrow{\phantom{aaa}} & \text{Central transversal theorem} & \xRightarrow{\phantom{aaa}} & \text{Ham sandwich}\\
        (k=0) & & (0\le k \le d-1) & & (k=d-1)\\
    \end{tikzcd}
\end{equation*}
    \caption{Implication of geometric transversal theorems.}
    \label{fig:implications}
\end{figure}


\noindent Furthermore, we prove a complex analogue of this result for the existence of a $\C$ $k$-transversal to a family of convex sets in $\C^d$, extending the work of McGinnis \cite{mcginnis2023complex} for $k=d-1$. Here, a $\C$ $k$-transversal is a complex $k$-dimensional affine subspace of $\C^d$ that intersects each set in $\F$. {The condition is a complex version from Definition \ref{def:R-dependency-consistent-with-tuples}, which is described in Definition \ref{def:dependency-consistent-with-tuples}.}

\begin{theorem} \label{thm:main-C}
    Let $0 \le k < d$ be integers and $\F$ a finite family of convex sets in $\C^d$. Then, there exists a $\C$ $k$-transversal to $\F$ if and only if $\F$ is $\C$-dependency consistent with $(d-k)$-tuples in a finite set of points $P \subseteq \C^k$.
\end{theorem}

\subsection*{Overview}

In Section \ref{sec:main-result}, we treat both Theorem \ref{thm:main-R} and Theorem \ref{thm:main-C} in a unifying way. Their proofs are topological and rely on a configuration space -- test map scheme. The key ingredient is a Borsuk-Ulam-type result on Stiefel manifolds, which was proved in the real case by Chan, Chen, Frick \& Hull \cite{chan2020borsuk} and in the complex case by Sadovek \& Sober\'on \cite{sadovek2024complex}.

In Section \ref{sec:application}, we apply Theorem \ref{thm:main-R} to give another proof of the central transversal theorem due to \v{Z}ivaljevi\'c \& Vre\'cica \cite{zivaljevic1990extension} and Dol'nikov \cite{Dolnikov:1992ut}, a staple result in discrete geometry. Similarly, Theorem \ref{thm:main-C} implies the recent complex analogue of the central transversal theorem due to Sadovek \& Sober\'on \cite{sadovek2024complex} by a completely analogous argument, giving a complex version of the diagram of implications in Figure \ref{fig:implications}. Thus, our main theorem, Theorem \ref{thm:main-F}, can be seen as a common generalization of the central transversal theorem and its complex analogue.

\section{The Goodman-Pollack-Wenger theorem and a complex analogue}
\label{sec:g-p-w-(d-1)}

In this section, we recall prior work on the existence of a $(d-1)$-transversal in both the real and the complex case, which sets the stage for the formulation of our main result in Section \ref{sec:main-result}.

Namely, to pass to a higher dimension setting, the linear ordering in Hadwiger's theorem is replaced by the notion of \textit{order type} of a set of points in $\mathbb{R}^{k}$ for some $0\leq k\leq d-1$ in the Goodman-Pollack-Wenger theorem. However, we will state the Goodman-Pollack-Wenger theorem as it appears in \cite{HolmsenChapter}, which is equivalent to the original formulation in \cite{PollackNecessary1990}. The following definition will be needed.

\begin{definition}\label{def:SepConsistent}
    Let $\F$ be a finite family of convex sets in $\mathbb{R}^d$. We say that $\F$ \textit{separates consistently} with a set of points $P\subseteq \mathbb{R}^k$ if there exists a map $\phi: \F \rightarrow P$ such that for any two subfamilies $\F_1$, $\F_2$ with $|\F_1| + |\F_2| \leq k+2$
    \[
    \conv(\F_1) \cap \conv(\F_2) = \emptyset \implies \conv(\phi(\F_1))\cap \conv(\phi(\F_2)) = \emptyset.
    \]
\end{definition}

We note that it is a consequence of the well-known Kirchberger's theorem \cite{KirchbergerTheorem} that the condition $|\F_1| + |\F_2| \leq k+2$ in Definition \ref{def:SepConsistent} could be removed and it would still be an equivalent definition. The statement of the Goodman-Pollack-Wenger theorem can be stated as follows:

\begin{theorem}[Goodman-Pollack-Wenger \cite{PollackNecessary1990}]\label{thm:GPW}
    A finite family of convex sets $\F$ in $\mathbb{R}^d$ has a $(d-1)$-transversal if and only if $\F$ separates consistently with a set $P\subseteq \mathbb{R}^{d-1}$.
\end{theorem}

As outlined in \cite{mcginnis2023necessary}, Definition \ref{def:SepConsistent} has the following equivalent linear-algebraic formulation. We present it as it is more similar to the condition present in our main result.

\begin{proposition}[{\cite{mcginnis2023necessary}}]\label{prop:Dependency}
    A finite family of convex sets $\F$ in $\mathbb{R}^d$ separates consistently with a set of points $P\subseteq \mathbb{R}^k$ if and only if there exists a map $\phi: \F \rightarrow P$ such that for any subfamily $\F'$ with $|\F'|\leq k+2$ and 
    for any nontrivial affine dependence 
\[
\sum_{F\in \F'} a_F=0,\, \sum_{F\in \F'} a_F\phi(F)= 0,
\]
there exist points $q_F\in F$ and real numbers $r_F\geq 0$ such that 
\[
\sum_{F\in \F'} r_Fa_F=0,\, \sum_{F\in \F'} (r_Fa_F)q_F= 0
\]
is an affine dependence of the points $q_F$ and the numbers $r_Fa_F$ are not all $0$. 
\end{proposition}

An adaptation of the condition in Proposition \ref{prop:Dependency} to complex dependencies was used in \cite{mcginnis2023necessary} to formulate the key definition in the complex analogue of the Goodman-Pollack-Wenger theorem due to McGinnis. In what follows, a convex set in $\mathbb{C}^d$ is understood to be convex in the usual sense, namely, that for any two points in the set the real line segment between them is contained in the set. 

\begin{definition}\label{def:depConsistent}
Let $\F$ be a finite family of convex sets in $\mathbb{C}^d$, and let $P\subseteq \mathbb{C}^k$. We say that $\F$ is \textit{dependency-consistent} with $P$ if there exists a map $\phi:\F \rightarrow P$ such that for every subfamily $\F'$ with $|\F'|\leq 2k+3$ and every affine dependence
\[
\sum_{F\in \F'}a_F=0,\, \sum_{F\in \F'}a_F\phi(F)=0
\]
for complex numbers $a_F$, there exist real numbers $r_F\geq 0$ and points $q_F\in F$ for $F\in \F'$ such that
\[
\sum_{F\in \F'}r_Fa_F=0,\, \sum_{F\in \F'}(r_Fa_F)q_F=0
\]
where not all of the values $r_Fa_F$ are 0.
\end{definition}

\begin{theorem}[\cite{mcginnis2023complex}]\label{thm:ComplexGPW}
A finite family of convex sets $\F$ in $\mathbb{C}^{d}$ has a complex $(d-1)$-transversal if and only if $\F$ is dependency-consistent with a set $P\subseteq \mathbb{C}^{d-1}$.
\end{theorem}

\section{The main result}
\label{sec:main-result}

The following definition is the key notion that is needed to state our main result.

\begin{definition} \label{def:dependency-consistent-with-tuples}
    Let $0 \le k < d$ be integers, $\FF \in \{\R, \C\}$ a field, $\F$ a finite family of convex sets in $\FF^d$, and $P \subseteq \FF^k$ a finite set of points.
    We say that $\F$ is \emph{$\FF$-dependency consistent with $(d-k)$-tuples} in $P$ if there is a map $\phi \colon \F \to P$ such that for any subset $\F' \subseteq \F$ with $|\F'| \leq (k+1)(d-k)\dim_\R\FF +1$ and any $d-k$ affine dependencies
    \[
        \sum_{F \in \F'} a^{(i)}_F = 0 \in \FF, \sum_{F \in \F'} a^{(i)}_F \phi(F) = 0 \in \FF^k, \hspace{11mm} \textrm{for }i=1, \dots, d-k,
    \]
    which are not all trivial,
    there exist points $q_F \in F$ and real numbers $r_F \ge 0$, such that the affine dependencies
    \[
        \sum_{F \in \F'} r_Fa^{(i)}_F = 0 \in \FF, \sum_{F \in \F'} r_Fa^{(i)}_F q_F = 0 \in \FF^d, \hspace{7mm} \textrm{for }i=1, \dots, d-k,
    \]
    are not all trivial{, i.e. some $r_Fa^{(i)}_F$ is nonzero}.
\end{definition}

\begin{remark} \label{rem:k=0-and-k=d-1}
    Inserting $k=0$ in the real case of the previous definition, we obtain precisely the condition in Helly's theorem:
    \begin{itemize}
        \item Assuming that $\F$ is $\R$-dependency consistent with $d$-tuples in $P=\{0\}=\R^0$, Helly's condition follows by choosing, for a given subfamily $\F'=\{F_0, \dots, F_d\} \subseteq \F$, the coefficients in the $i$'th affine dependency to be
        \[
            a_{F_0}^{(i)} = 1,~a_{F_i}^{(i)} = -1, \hspace{3mm} \textrm{and} \hspace{3mm} a_{F_j}^{(i)} = 0, \hspace{3mm} \textrm{for}~j\neq 0,i.
        \]
        {By the condition on $\F$, there are real numbers $r_{F_i} \ge 0$ and points $q_{F_i} \in F_i$ such that
        \[
            r_{F_0}-r_{F_i} = 0 \in \R,~r_{F_0}q_{F_0}-r_{F_i}q_{F_i} = 0 \in \R^d, \hspace{3mm} \text{for}~i=1, \dots, d,
        \]
        are affine dependencies which are not all trivial. Consequently, $r_{F_0} = \dots = r_{F_d}$ are all strictly positive and $q_{F_0} = \dots =q_{F_d}$ is a common point of sets in $\F'$.}
        \item Assuming Helly's condition on $\F$, it follows that $\F$ is $\R$-dependency consistent with $d$-tuples in $P=\{0\}=\R^0$ by choosing $q_F$ to be the common point of the sets in $\F'$ and $r_F=1$, for each $F \in \F'$.
    \end{itemize}
    On the other hand, putting $k=d-1$ in Definition \ref{def:dependency-consistent-with-tuples}, we recover the Goodman-Pollack-Wenger condition as presented in Proposition \ref{prop:Dependency} in the real case and the complex analogue from Definition \ref{def:depConsistent}.
\end{remark}

We will denote by $W_n(\FF^d)$ the Stiefel manifold of $\FF$-orthonormal $n$-frames in $\FF^d$, for a field $\FF \in \{\R, \C\}$. As a key ingredient for our proof, we will employ the following Borsuk-Ulam-type theorem. The real case is due to Chan, Chen, Frick \& Hull \cite{chan2020borsuk}, while the complex case was proved by Sadovek \& Sober\'on \cite{sadovek2024complex}.

\begin{theorem}[\cites{chan2020borsuk,sadovek2024complex}] \label{thm:map-from-stiefel}
    Let $0 \le n < d$ be integers and $\FF \in \{\R, \C\}$ a field. Then, every continuous $\Z_2^n$-equivariant map
    \[
        W_n(\FF^d) \longrightarrow \FF^{d-1} \oplus \FF^{d-2} \oplus \dots \oplus \FF^{d-n}
    \]
    has the origin in its image, where the group acts by product antipodal action on both spaces.
\end{theorem}

The theorem above, in both the real and complex case, follows from computations of the appropriate Fadell-Husseini indices \cite{Fadell:1988tm}.

We are now ready to prove the main result of the paper. 

\begin{theorem} \label{thm:main-F}
    Let $0 \le k < d$ be integers, $\FF \in \{\R, \C\}$ a field, and $\F$ a finite family of convex sets in $\FF^d$. Then, $\F$ is $\FF$-dependency consistent with $(d-k)$-tuples in a finite set of points in $\FF^k$ if and only if there exists a $k$-dimensional $\FF$-transversal to $\F$.
\end{theorem}
\begin{proof} {If $\F$ has an $\FF$ $k$-transversal $T$, then for each $F\in \F$, we may choose points $q_F\in F\cap T$. Then, for any $\FF$ affine isomorphism $\psi:T\rightarrow \FF^k$, $\F$ is $\FF$-dependency consistent with $(d-k)$-tuples in $\{\psi(q_F)\}_{F\in \F}$. This is because the $q_F$ and the $\psi(q_F)$ satisfy the same affine dependencies.}

{For the ``only if'' direction}, let us assume that $\F$ does not admit a $k$-dimensional $\FF$-transversal.
We may also assume the sets of $\F$ are compact and we put them in a copy of $\FF^d$ lying in $\FF^d\! +\! e_{d+1}\!\! \subseteq\! \FF^{d+1}$, where $e_{d+1} \coloneqq (0, \dots, 0, 1) \in \FF^{d+1}$. Throughout the proof, we will assume the inner products to be with respect to the field $\FF$, unless explicitly stated otherwise. {In particular, for $\FF = \C$, this inner product is given by $\langle (z_1, \dots, z_n), (w_1, \dots, w_n) \rangle = \sum_{i=1}^n z_i \overline{w}_i$.}
For each $(d-k)$-orthonormal frame $(v_1,\dots,v_{d-k})\in W_{d-k}(\FF^{d+1})$, we set 
\[
    V \coloneqq \linspan_\FF\{v_1,\dots,v_{d-k}\} \subseteq \FF^{d+1}
\]
and denote by 
\[
    \proj_V \colon~ \FF^{d+1} \longrightarrow V,~ x \longmapsto \langle x, v_1 \rangle v_1 + \dots + \langle x, v_{d-k} \rangle v_{d-k}
\]
the orthogonal projection. For each $F\in \F$, let $p_{V,F}$ be the point of $\proj_V(F) \subseteq V$ closest to the origin. Such a point is unique since $\proj_V(F)$ is convex.

Let $P \subseteq \FF^k$ and $\phi \colon \F \to P$ be according to Definition \ref{def:dependency-consistent-with-tuples}.
Define a $\Z_2^{d-k}$-equivariant test map as
\begin{align*}
    W_{d-k}(\FF^{d+1}) &\longrightarrow (\FF^{k+1})^{d-k}\\
    (v_1, \dots, v_{d-k}) &\longmapsto \Big(\sum_{F\in \F} \langle v_i, p_{V,F} \rangle, \sum_{F\in \F} \langle v_i, p_{V,F} \rangle \phi(F)\Big)_{i=1}^{d-k}
\end{align*}
where the group has the product-antipodal action on both spaces.
By Theorem \ref{thm:map-from-stiefel}, there exists a $(d-k)$-frame $({u}_1, \dots, {u}_{d-k}) \in W_{d-k}(\FF^{d+1})$ that maps to the origin.

First, we observe that $e_{d+1}\! \notin {U\coloneqq \linspan_\FF\{u_1,\dots,u_{d-k}\}}$, since otherwise $\proj_{{U}}(F)\subseteq \FF^d \!+\! e_{d+1}$ and, in particular, $p_{{U},F}\!\in \FF^d\! +\! e_{d+1}$, for all $F \in \F$. The latter would imply
\[
    0 \neq\sum_{F\in \F} p_{{U},F} = \sum_{i=1}^{d-k} \sum_{F\in \F} \overline{\langle {u}_i, p_{{U},F} \rangle} {u}_i,
\]
which contradicts the fact that $\sum_{F\in \F} \overline{\langle {u}_i, p_{{U},F} \rangle} = 0$, for every $i=1, \dots, d-k$.

Next, we may assume that there are sets $F\in \F$ for which $0\notin \proj_{{U}}(F)$, for otherwise ${U}^\perp\! \cap (\FF^d\! +\! e_{d+1})$ would be a $k$-transversal to $\F$, contradicting our initial assumption. Indeed, $0\in \proj_{{U}}(F)$ is equivalent to ${U}^\perp\cap F \neq \emptyset$, and $e_{d+1}\notin {U}$ implies that ${U}^\perp \cap (\FF^d + e_{d+1})$ is a $k$-flat in $\FF^d + e_{d+1}$.

Therefore, the subfamily $\G \coloneqq \{F \in \F \colon~ 0\notin \proj_{{U}}(F)\}$ is non-empty and for each $F \in \G$ there are values $\langle {u}_i, p_{{U},F} \rangle$ which are nonzero. We have that the origin in $\FF^{(k+1)(d-k)}$ is in the convex hull of the set of points
\[
        \Big\{\Big(\langle {u}_i, p_{{U},F} \rangle, \langle {u}_i, p_{{U},F} \rangle \phi(F)\Big)_{i=1}^{d-k} \in \FF^{(k+1)(d-k)} \colon~ F \in \G\Big\}.
\]
Thus, by Carath\'eodory's theorem, there exists a subfamily $\F' \subseteq \G$ of size at most $(k+1)(d-k)\dim_\R\FF+1$ and positive real numbers $a_F$ that sum up to one such that the affine dependencies
\begin{equation*}
    \sum_{F\in \F'} a_F \langle u_i, p_{{U},F} \rangle = 0  \,\, {\rm and} \,\, \sum_{F\in \F'} a_F \langle {u}_i, p_{{U},F} \rangle \phi(F) = 0, \hspace{7mm} \textrm{for }i=1, \dots, d-k,
\end{equation*}
are not all trivial.
Since $\F$ is $\FF$-dependency consistent with $(d-k)$-tuples in $P$, there exist points $q_F\in F$ and real numbers $r_F \ge 0$, for $F\in \F'$, such that the affine dependencies 
\begin{equation*}
    \sum_{F\in \F'} r_Fa_F \langle {u}_i, p_{{U},F} \rangle = 0  \,\, {\rm and} \,\, \sum_{F\in \F'} r_Fa_F \langle {u}_i, p_{{U},F} \rangle q_F=0, \hspace{7mm} \textrm{for }i=1, \dots, d-k,
\end{equation*}
are not all trivial. After projecting each $q_F$ onto ${U}$, we get that
\begin{align*}
    \sum_{F\in \F'} r_Fa_F\langle  \proj_{{U}}(q_F), p_{{{U}},F} \rangle &= \sum_{F\in \F'} r_Fa_F\sum_{i=1}^{d-k} \langle \proj_{{U}}(q_F), {u}_i \rangle \langle {u}_i, p_{{{U}},F} \rangle \notag\\
    &= \sum_{i=1}^{d-k} \left\langle \sum_{F\in \F'} r_Fa_F \langle {u}_i, p_{{{U}},F} \rangle \proj_{{U}}(q_F), {u}_i \right\rangle = 0.
\end{align*}
However, for all $F\in \F'$, we have $\mathrm{Re}\langle  \proj_{{U}}(q_F), p_{{{U}},F} \rangle = \langle \proj_{{U}}(q_F), p_{{{U}},F} \rangle_\R > 0$. This is because $0 \neq p_{{{U}},F} \in \proj_{{U}}(F)$ is the closest point to the origin of the convex set $\proj_{{U}}(F)$, which also contains $\proj_{{U}}(q_F)$. 
Therefore, $\proj_{{U}}(q_F)$ lies on the positive side of the real affine hyperplane
\[
    (p_{{{U}},F})^{\perp_\R} + p_{{{U}},F} \subseteq {{U}} \cong (\R^{\dim_\R \FF})^{d-k},
\]
and consequently $\langle \proj_{{U}}(q_F), p_{{{U}},F} \rangle_\R > 0$. See Figure \ref{fig:proj-of-F} for an illustration. This leads to a contradiction and completes the proof.
\end{proof}

\begin{figure}[h]
  \centering
  \includegraphics[width=0.4\textwidth]{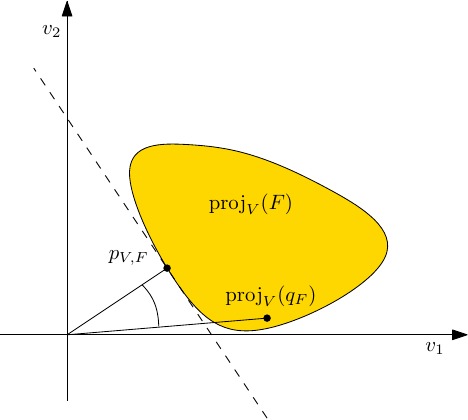}
  \caption{The point $\proj_{{U}}(q_F)$ lies above the dashed line $(p_{{{U}},F})^{\perp_\R}\!+\!p_{{{U}},F}$.}
  \label{fig:proj-of-F}
\end{figure}

\section{Another proof of the central transversal theorem}\label{sec:application}

As an application of Theorem \ref{thm:main-F}, we provide another proof of the \textit{central transversal theorem} due to \v{Z}ivaljevi\'c \& Vre\'cica \cite{zivaljevic1990extension} and Dol'nikov \cite{Dolnikov:1992ut}. We attribute this observation to Andreas Holmsen who shared this application of the main theorem to us via personal communication. {To this end, we first reprove Dol'nikov's more general result on transversals (and extend it to the complex case)}, which is the main application of Theorem \ref{thm:main-F}.

\begin{lemma}[{Dol'nikov \cite{dol1993transversals} for $\FF = \R$}]\label{lem:center transversal}
    Let $\F_0,\F_1,\dots,\F_k$ be finite families of convex sets in $\mathbb{F}^d$ such that for each $j=0, \dots, k$, every $\dim_\R(\FF)(d-k)+1$ or fewer sets in $\F_j$ have a nonempty intersection. Then the family $\F_0 \cup \dots \cup \F_k$ has an $\FF$ $k$-transversal.
\end{lemma}
\begin{proof}
    Let $P=\{p_0,\dots,p_k\}\subseteq \mathbb{F}^k$ be a set of points which are $\FF$ affinely independent and $\F = \F_0 \cup \dots \cup \F_k$. Define the map $\phi: \F \to \mathbb{F}^k$ by $F \mapsto p_j$ for each $F \in \F_j$. If there are multiple options for $j$, we choose one arbitrarily. We will show that $\F$ is $\FF$-dependency consistent with $(d-k)$-tuples in $P$ under the map $\phi$, and hence the result follows by Theorem \ref{thm:main-F}.

    Let $\F'\subseteq \F$ be a subfamily of size at most $\dim_\R(\FF)(k+1)(d-k)+1$, and let
    \[
        \sum_{F \in \F'} a^{(i)}_F = 0 \in \FF, \sum_{F \in \F'} a^{(i)}_F \phi(F) = 0 \in \FF^k, \hspace{11mm} \textrm{for }i=1, \dots, d-k,
    \]
    be $\FF$ affine dependencies, which are not all trivial. Since the points of $P$ are in general position, we have
    \[
    \sum_{F\in \F' \cap \phi^{-1}(p_j)}a^{(i)}_F = 0 \in \FF, \hspace{11mm} \text{for each } i=1, \dots, d-k,~ \text{and each }~ j=0, \dots, k.
    \]
    In particular, we have that the set of points 
    \[
        \{(a^{(1)}_F,\dots,a^{(d-k)}_F) \colon F\in \F' \cap \phi^{-1}(p_j)\}\subseteq \F^{d-k}, \hspace{11mm} \text{for each } ~ j=0, \dots, k,
    \]
    contains the origin in its convex hull. Let $0\leq j_0\leq k$ be such that the corresponding set above is not a singleton consisting only of the origin in $\F^{d-k}$. By Carath\'eodory's theorem, there are sets $F_1,\dots,F_s \in \F' \cap \phi^{-1}(p_{j_0})$ and positive real numbers $r_1,\dots,r_s$, where $1 \le s\leq \dim_\R(\FF)(d-k)+1$, such that 
    \[
        \sum_{i=1}^s r_i (a^{(1)}_{F_i},\dots,a^{(d-k)}_{F_i}) = 0 \in \FF^{d-k}
    \]
    and none of the points $(a^{(1)}_{F_i},\dots,a^{(d-k)}_{F_i})$ are the origin in $\FF^{d-k}$.

    Now, from the assumption of the lemma, there exists a point $q\in  F_1 \cap \dots \cap F_s$. For $F\in \F'$, we define a point $q_F \in F$ and a real number $r_F \ge 0$ as 
    \[
    q_F \coloneqq \begin{cases} q & \textrm{if } F=F_i,\\
    \text{any point in}~F & \textrm{otherwise}\end{cases} \hspace{5mm} \text{and} \hspace{5mm}
    r_F \coloneqq \begin{cases} r_i & \textrm{if } F=F_i,\\
    0 & \textrm{otherwise}.\end{cases}
    \]
    Then we have the following affine dependencies
    \[
        \sum_{F \in \F'} r_Fa^{(i)}_F = 0 \in \FF, \sum_{F \in \F'} r_Fa^{(i)}_F q_F = 0 \in \FF^k, \hspace{11mm} \textrm{for }i=1, \dots, d-k,
    \]
    which are not all trivial. Thus, we conclude that $\F$ is $\F$-dependency consistent with $(d-k)$-tuples in $P$ under the map $\phi$, which finishes the proof of the lemma.
\end{proof}

We can now prove the central transversal theorem. We state and prove it in its discrete form \cite[Chapter 1.4]{matousek2002lectures}.

\begin{theorem}[\v{Z}ivaljevi\'c-Vre\'cica \cite{zivaljevic1990extension}, Dol'nikov \cite{Dolnikov:1992ut}] \label{thm:central-transversal}
    Let $A_0,\dots,A_k$ be finite point sets in $\mathbb{R}^d$. Then there exists a $k$-dimensional affine subspace $T \subseteq \R^d$ such that every closed halfspace $H \subseteq \R^d$ that contains $T$ satisfies
    \begin{equation*}
        |H \cap A_j| \ge \frac{1}{d-k+1}|A_j|, \hspace{5mm} \text{for each}~ j=0, \dots, k.
    \end{equation*}
\end{theorem}
\begin{proof}
    For each $j=0, \dots, k$, let $\F_j$ be the family of convex sets given by the convex hulls of subsets of $A_j$ of size greater than $\frac{d-k}{d-k+1}|A_j|$. By a simple counting argument, every $d-k+1$ of fewer sets in $\F_j$ have a nonempty intersection. (In fact, the intersection contains a point from $A_j$.) By Lemma \ref{lem:center transversal}, there is a $k$-transversal $T$ to $\F_0 \cup \dots \cup \F_k$.  If a halfspace containing $T$ contains fewer than $\frac{1}{d-k+1}|A_j|$ points of $A_j$, then its complement contains more than $\frac{d-k}{d-k+1}|A_j|$ points of $A_j$, contradicting the fact that $T$ is a transversal to $\F_j$.
\end{proof}

{\v{Z}ivaljevi\'c and Vre\'cica \cite{zivaljevic1990extension} originally proved the central transversal theorem (CTT) by first establishing an appropriate result about the zeros of sections of certain bundles over the Grassmannian. Later Manta and Sober\'on \cite{manta2024generalizations} proved CTT by directly employing the real case of Theorem \ref{thm:map-from-stiefel} on equivariant maps from Stiefel manifolds. The complex case of the latter was then used by Sadovek and Sober\'on \cite{sadovek2024complex} to obtain the complex analogue of CTT (see Theorem \ref{thm:cplx-central-transv} below). Our proof of the more general Lemma \ref{lem:center transversal} uses our main theorem, Theorem \ref{thm:main-F}, as a black box and hence indirectly follows by the same result on Stiefel manifolds (Theorem \ref{thm:map-from-stiefel}). This is in contrast with its original proof by Dol'nikov \cite{dol1993transversals}, which again followed by establishing a vanishing result on sections on a bundle over Grassmannian.}

{As mentioned, we provide the complex analogue of the central transversal theorem, which follows readily from Lemma \ref{lem:center transversal}. We again state it its discretized form.}

\begin{theorem}[Sadovek-Sober\'on \cite{sadovek2024complex}] \label{thm:cplx-central-transv}
    Let $A_0,\dots,A_k$ be finite point sets in $\C^d$.
    Then there exists a complex $k$-dimensional affine subspace $T \subseteq \C^d$ such that every closed real halfspace $H \subseteq \C^d$ that contains $T$ satisfies
    \begin{equation*}
        |H \cap A_j| \ge \frac{1}{2d-2k+1}|A_j|, \hspace{5mm} \text{for each}~ j=0, \dots, k.
    \end{equation*}
\end{theorem}

\section{Concluding remarks}

\begin{itemize}
    \item The condition in Definition \ref{def:dependency-consistent-with-tuples} does not seem to have a more combinatorial formulation as the condition in Definition \ref{def:SepConsistent}. However, this is likely necessary since in \cite{holmsen2004nohelly} it was shown that there is no Hadwiger theorem for 1-transversals in $\mathbb{R}^3$, and thus a condition based solely on the separation properties of a set of points in $\R$ would not suffice.

    \item A colorful generalization of the Goodman-Pollack-Wenger theorem was recently proven in \cite{cheong2024new}. We leave it as an open problem to determine if there is an analogous colorful generalization of Theorem \ref{thm:main-F}. 

    \item It would be interesting to see if Theorem \ref{thm:main-F} or the proof method can be applied to existing problems in geometric transversal theory. For example, it is known that if a family of pairwise disjoint unit balls in $\R^d$ with a linear ordering such that every $2d$ balls has a $1$-transversal consistent with the order, then the family has a $1$-transversal \cites{borcea2008Line,cheong2008Helly}. Furthermore, the constant $2d$ cannot be lowered to $2d-2$ and it is unknown if it can be lowered to $2d-1$ \cite{cheong2012lower}. When $k=1$ and $\FF = \R$, the constant in Theorem \ref{thm:main-F} is $2d-1$, so perhaps Theorem \ref{thm:main-F} or our proof method has some bearing on this problem.
\end{itemize}

\section{Acknowledgements}

We would like to thank the anonymous referees for their helpful suggestions to improve the exposition of the manuscript and Florian Frick, Andreas Holmsen, and Pablo Sober\'on for their feedback on a first version of this paper. We additionally thank Andreas Holmsen for sharing his observation that our main results imply the central transversal theorems.

\bibliographystyle{plain}
\bibliography{references.bib}

\end{document}